\documentclass[11pt,fleqn]{article}

\usepackage{amsmath,amssymb,amsthm}
\usepackage{geometry}
\usepackage[pdfpagemode=UseNone,pdfstartview=FitH,hypertex,bookmarks=false]{hyperref}

\newcommand{\bdry}[1]{\partial #1}
\newcommand{\A}{{\cal A}}
\newcommand{\F}{{\cal F}}
\newcommand{\calN}{{\cal N}}
\newcommand{\closure}[1]{\overline{#1}}
\newcommand{\dint}{\ds{\int}}
\newcommand{\dist}[2]{\text{dist}\, (#1,#2)}
\newcommand{\ds}[1]{\displaystyle #1}
\newcommand{\eps}{\varepsilon}
\newcommand{\id}[1][]{id_{\, #1}}
\newcommand{\incl}{\subset}
\newcommand{\loc}{\text{loc}}
\newcommand{\M}{{\cal M}}
\newcommand{\N}{\mathbb N}
\newcommand{\norm}[2][]{\left\|#2\right\|_{#1}}
\newcommand{\PS}[1]{$(\text{PS})_{#1}$}
\newcommand{\R}{\mathbb R}
\newcommand{\RP}{\R \text{P}}
\newcommand{\restr}[2]{\left.#1\right|_{#2}}
\newcommand{\seq}[1]{\left(#1\right)}
\newcommand{\set}[1]{\left\{#1\right\}}
\newcommand{\Z}{\mathbb Z}

\DeclareMathOperator{\supp}{supp}

\newenvironment{enumroman}{\begin{enumerate}
\renewcommand{\theenumi}{$(\roman{enumi})$}
\renewcommand{\labelenumi}{$(\roman{enumi})$}}{\end{enumerate}}

\newenvironment{properties}[1]{\begin{enumerate}
\renewcommand{\theenumi}{$(#1_\arabic{enumi})$}
\renewcommand{\labelenumi}{$(#1_\arabic{enumi})$}}{\end{enumerate}}

\newtheorem{corollary}{Corollary}[section]
\newtheorem{lemma}[corollary]{Lemma}
\newtheorem{proposition}[corollary]{Proposition}
\newtheorem{theorem}[corollary]{Theorem}

\theoremstyle{remark}
\newtheorem{example}[corollary]{Example}
\newtheorem{remark}[corollary]{Remark}

\numberwithin{equation}{section}

\title{\bf $p$-Laplacian problems involving critical Hardy-Sobolev exponents}
\author{\bf Kanishka Perera\\
Department of Mathematical Sciences\\
Florida Institute of Technology\\
Melbourne, FL 32901, USA\\
\em kperera@fit.edu\\
[\bigskipamount]
\bf Wenming Zou\\
Department of Mathematical Sciences\\
Tsinghua University\\
Beijing 100084, China\\
\em wzou@math.tsinghua.edu.cn}
\date{}

\begin{document}

\maketitle

\begin{abstract}
We prove existence, multiplicity, and bifurcation results for $p$-Laplacian problems involving critical Hardy-Sobolev exponents. Our results are mainly for the case $\lambda \ge \lambda_1$ and extend results in the literature for $0 < \lambda < \lambda_1$. In the absence of a direct sum decomposition, we use critical point theorems based on a cohomological index and a related pseudo-index.
\end{abstract}

\section{Introduction}

Consider the critical $p$-Laplacian problem
\begin{equation} \label{1}
\left\{\begin{aligned}
- \Delta_p\, u & = \lambda\, |u|^{p-2}\, u + \frac{|u|^{p^\ast(s) - 2}}{|x|^s}\, u && \text{in } \Omega\\[10pt]
u & = 0 && \text{on } \bdry{\Omega},
\end{aligned}\right.
\end{equation}
where $\Omega$ is a bounded domain in $\R^N$ containing the origin, $1 < p < N$, $\lambda > 0$ is a parameter, $0 < s < p$, and $p^\ast(s) = (N - s)\, p/(N - p)$ is the critical Hardy-Sobolev exponent. In \cite{MR1695021}, Ghoussoub and Yuan showed, among other things, that this problem has a positive solution when $N \ge p^2$ and $0 < \lambda < \lambda_1$, where $\lambda_1 > 0$ is the first eigenvalue of the eigenvalue problem
\begin{equation} \label{2}
\left\{\begin{aligned}
- \Delta_p\, u & = \lambda\, |u|^{p-2}\, u && \text{in } \Omega\\[10pt]
u & = 0 && \text{on } \bdry{\Omega}.
\end{aligned}\right.
\end{equation}
In the present paper we mainly consider the case $\lambda \ge \lambda_1$. Our existence results are the following.

\begin{theorem} \label{Th4}
If $N \ge p^2$ and $0 < \lambda < \lambda_1$, then problem \eqref{1} has a positive ground state solution.
\end{theorem}

\begin{theorem} \label{Theorem 1}
If $N \ge p^2$ and $\lambda > \lambda_1$ is not an eigenvalue of problem \eqref{2}, then problem \eqref{1} has a nontrivial solution.
\end{theorem}

\begin{theorem} \label{Theorem 2}
If
\begin{equation} \label{27}
(N - p^2)(N - s) > (p - s)\, p
\end{equation}
and $\lambda \ge \lambda_1$, then problem \eqref{1} has a nontrivial solution.
\end{theorem}

\begin{remark}
We note that \eqref{27} implies $N > p^2$.
\end{remark}

\begin{remark}
In the nonsingular case $s = 0$, related results can be found in Degiovanni and Lancelotti \cite{MR2514055} for the $p$-Laplacian and in Mosconi et al.\! \cite{MR3530213} for the fractional $p$\nobreakdash-Laplacian.
\end{remark}

Weak solutions of problem \eqref{1} coincide with critical points of the $C^1$-functional
\[
I_\lambda(u) = \int_\Omega \left[\frac{1}{p}\, \big(|\nabla u|^p - \lambda\, |u|^p\big) - \frac{1}{p^\ast(s)}\, \frac{|u|^{p^\ast(s)}}{|x|^s}\right] dx, \quad u \in W^{1,p}_0(\Omega).
\]
Recall that $I_\lambda$ satisfies the Palais-Smale compactness condition at the level $c \in \R$, or the \PS{c} condition for short, if every sequence $\seq{u_j} \subset W^{1,p}_0(\Omega)$ such that $I_\lambda(u_j) \to c$ and $I_\lambda'(u_j) \to 0$ has a convergent subsequence. Let
\begin{equation} \label{3}
\mu_s = \inf_{u \in W^{1,p}_0(\Omega) \setminus \set{0}}\, \frac{\dint_\Omega |\nabla u|^p\, dx}{\left(\dint_\Omega \frac{|u|^{p^\ast(s)}}{|x|^s}\, dx\right)^{p/p^\ast(s)}}
\end{equation}
be the best constant in the Hardy-Sobolev inequality, which is independent of $\Omega$ (see \cite[Theorem 3.1.(1)]{MR1695021}). It was shown in \cite[Theorem 4.1.(2)]{MR1695021} that $I_\lambda$ satisfies the \PS{c} condition for all
\[
c < \frac{p - s}{(N - s)\, p}\; \mu_s^{(N-s)/(p-s)}
\]
for any $\lambda > 0$. We will prove Theorems \ref{Th4} -- \ref{Theorem 2} by constructing suitable minimax levels below this threshold for compactness. When $0 < \lambda < \lambda_1$, we will show that the infimum of $I_\lambda$ on the Nehari manifold is below this level. When $\lambda \ge \lambda_1$, $I_\lambda$ no longer has the mountain pass geometry and a linking type argument is needed. However, the classical linking theorem cannot be used here since the nonlinear operator $- \Delta_p$ does not have linear eigenspaces. We will use a nonstandard linking construction based on sublevel sets as in Perera and Szulkin \cite{MR2153141} (see also Perera et al.\! \cite[Proposition 3.23]{MR2640827}). Moreover, the standard sequence of eigenvalues of $- \Delta_p$ based on the genus does not give enough information about the structure of the sublevel sets to carry out this construction. Therefore, we will use a different sequence of eigenvalues introduced in Perera \cite{MR1998432} that is based on a cohomological index.

For $1 < p < \infty$, eigenvalues of problem \eqref{2} coincide with critical values of the functional
\[
\Psi(u) = \frac{1}{\dint_\Omega |u|^p\, dx}, \quad u \in \M = \set{u \in W^{1,p}_0(\Omega) : \int_\Omega |\nabla u|^p\, dx = 1}.
\]
Let $\F$ denote the class of symmetric subsets of $\M$, let $i(M)$ denote the $\Z_2$-cohomological index of $M \in \F$ (see section \ref{Section 2.1}), and set
\[
\lambda_k := \inf_{M \in \F,\; i(M) \ge k}\, \sup_{u \in M}\, \Psi(u), \quad k \in \N.
\]
Then $0 < \lambda_1 < \lambda_2 \le \lambda_3 \le \cdots \to \infty$ is a sequence of eigenvalues of \eqref{2} and
\begin{equation} \label{4}
\lambda_k < \lambda_{k+1} \implies i(\Psi^{\lambda_k}) = i(\M \setminus \Psi_{\lambda_{k+1}}) = k,
\end{equation}
where $\Psi^a = \set{u \in \M : \Psi(u) \le a}$ and $\Psi_a = \set{u \in \M : \Psi(u) \ge a}$ for $a \in \R$ (see Perera et al.\! \cite[Propositions 3.52 and 3.53]{MR2640827}). We also prove the following bifurcation and multiplicity results for problem \eqref{1} that do not require $N \ge p^2$. Set
\[
V_s(\Omega) = \int_\Omega |x|^{(N-p)\, s/(p-s)}\, dx,
\]
and note that
\begin{equation} \label{28}
\int_\Omega |u|^p\, dx \le V_s(\Omega)^{(p-s)/(N-s)} \left(\int_\Omega \frac{|u|^{p^\ast(s)}}{|x|^s}\, dx\right)^{p/p^\ast(s)} \quad \forall u \in W^{1,p}_0(\Omega)
\end{equation}
by the H\"{o}lder inequality.

\begin{theorem} \label{Theorem 4}
If
\[
\lambda_1 - \frac{\mu_s}{V_s(\Omega)^{(p-s)/(N-s)}} < \lambda < \lambda_1,
\]
then problem \eqref{1} has a pair of nontrivial solutions $\pm\, u^\lambda$ such that $u^\lambda \to 0$ as $\lambda \nearrow \lambda_1$.
\end{theorem}

\begin{theorem} \label{Theorem 5}
If $\lambda_k \le \lambda < \lambda_{k+1} = \cdots = \lambda_{k+m} < \lambda_{k+m+1}$ for some $k, m \in \N$ and
\begin{equation} \label{1.10}
\lambda > \lambda_{k+1} - \frac{\mu_s}{V_s(\Omega)^{(p-s)/(N-s)}},
\end{equation}
then problem \eqref{1} has $m$ distinct pairs of nontrivial solutions $\pm\, u^\lambda_j,\, j = 1,\dots,m$ such that $u^\lambda_j \to 0$ as $\lambda \nearrow \lambda_{k+1}$.
\end{theorem}

In particular, we have the following existence result that is new when $N < p^2$.

\begin{corollary} \label{Corollary 1}
If
\[
\lambda_k - \frac{\mu_s}{V_s(\Omega)^{(p-s)/(N-s)}} < \lambda < \lambda_k
\]
for some $k \in \N$, then problem \eqref{1} has a nontrivial solution.
\end{corollary}

\begin{remark}
We note that $\lambda_1 \ge \mu_s/V_s(\Omega)^{(p-s)/(N-s)}$. Indeed, let $\varphi_1 > 0$ be an eigenfunction associated with $\lambda_1$. Then
\[
\lambda_1 = \frac{\dint_\Omega |\nabla \varphi_1|^p\, dx}{\dint_\Omega \varphi_1^p\, dx} \ge \frac{\mu_s \left(\dint_\Omega \frac{\varphi_1^{p^\ast(s)}}{|x|^s}\, dx\right)^{p/p^\ast(s)}}{\dint_\Omega \varphi_1^p\, dx} \ge \frac{\mu_s}{V_s(\Omega)^{(p-s)/(N-s)}}
\]
by \eqref{3} and \eqref{28}.
\end{remark}

\begin{remark}
Since $V_0(\Omega)$ is the volume of $\Omega$, in the nonsingular case $s = 0$, Theorems \ref{Theorem 4} \& \ref{Theorem 5} and Corollary \ref{Corollary 1} reduce to Perera et al.\! \cite[Theorem 1.1 and Corollary 1.2]{PeSqYa1}, respectively.
\end{remark}

\section{Preliminaries}

\subsection{Cohomological index} \label{Section 2.1}

The $\Z_2$-cohomological index of Fadell and Rabinowitz \cite{MR57:17677} is defined as follows. Let $W$ be a Banach space and let $\A$ denote the class of symmetric subsets of $W \setminus \set{0}$. For $A \in \A$, let $\overline{A} = A/\Z_2$ be the quotient space of $A$ with each $u$ and $-u$ identified, let $f : \overline{A} \to \RP^\infty$ be the classifying map of $\overline{A}$, and let $f^\ast : H^\ast(\RP^\infty) \to H^\ast(\overline{A})$ be the induced homomorphism of the Alexander-Spanier cohomology rings. The cohomological index of $A$ is defined by
\[
i(A) = \begin{cases}
0 & \text{if } A = \emptyset\\[5pt]
\sup \set{m \ge 1 : f^\ast(\omega^{m-1}) \ne 0} & \text{if } A \ne \emptyset,
\end{cases}
\]
where $\omega \in H^1(\RP^\infty)$ is the generator of the polynomial ring $H^\ast(\RP^\infty) = \Z_2[\omega]$.

\begin{example}
The classifying map of the unit sphere $S^{m-1}$ in $\R^m,\, m \ge 1$ is the inclusion $\RP^{m-1} \incl \RP^\infty$, which induces isomorphisms on the cohomology groups $H^q$ for $q \le m - 1$, so $i(S^{m-1}) = m$.
\end{example}

The following proposition summarizes the basic properties of this index.

\begin{proposition}[Fadell-Rabinowitz \cite{MR57:17677}]
The index $i : \A \to \N \cup \set{0,\infty}$ has the following properties:
\begin{properties}{i}
\item Definiteness: $i(A) = 0$ if and only if $A = \emptyset$.
\item Monotonicity: If there is an odd continuous map from $A$ to $B$ (in particular, if $A \subset B$), then $i(A) \le i(B)$. Thus, equality holds when the map is an odd homeomorphism.
\item Dimension: $i(A) \le \dim W$.
\item Continuity: If $A$ is closed, then there is a closed neighborhood $N \in \A$ of $A$ such that $i(N) = i(A)$. When $A$ is compact, $N$ may be chosen to be a $\delta$-neighborhood $N_\delta(A) = \set{u \in W : \dist{u}{A} \le \delta}$.
\item Subadditivity: If $A$ and $B$ are closed, then $i(A \cup B) \le i(A) + i(B)$.
\item Stability: If $SA$ is the suspension of $A \ne \emptyset$, obtained as the quotient space of $A \times [-1,1]$ with $A \times \set{1}$ and $A \times \set{-1}$ collapsed to different points, then $i(SA) = i(A) + 1$.
\item Piercing property: If $A$, $A_0$ and $A_1$ are closed, and $\varphi : A \times [0,1] \to A_0 \cup A_1$ is a continuous map such that $\varphi(-u,t) = - \varphi(u,t)$ for all $(u,t) \in A \times [0,1]$, $\varphi(A \times [0,1])$ is closed, $\varphi(A \times \set{0}) \subset A_0$ and $\varphi(A \times \set{1}) \subset A_1$, then $i(\varphi(A \times [0,1]) \cap A_0 \cap A_1) \ge i(A)$.
\item Neighborhood of zero: If $U$ is a bounded closed symmetric neighborhood of the origin, then $i(\bdry{U}) = \dim W$.
\end{properties}
\end{proposition}

\subsection{Abstract critical point theorems}

We will prove Theorems \ref{Theorem 1} and \ref{Theorem 2} using the following abstract critical point theorem proved in Yang and Perera \cite{YaPe2}, which generalizes the well-known linking theorem of Rabinowitz \cite{MR0488128}.

\begin{theorem} \label{Theorem 3}
Let $I$ be a $C^1$-functional defined on a Banach space $W$, and let $A_0$ and $B_0$ be disjoint nonempty closed symmetric subsets of the unit sphere $S = \set{u \in W : \norm{u} = 1}$ such that
\[
i(A_0) = i(S \setminus B_0) < \infty.
\]
Assume that there exist $R > r > 0$ and $v \in S \setminus A_0$ such that
\[
\sup I(A) \le \inf I(B), \qquad \sup I(X) < \infty,
\]
where
\begin{gather*}
A = \set{tu : u \in A_0,\, 0 \le t \le R} \cup \set{R\, \pi((1 - t)\, u + tv) : u \in A_0,\, 0 \le t \le 1},\\[10pt]
B = \set{ru : u \in B_0},\\[10pt]
X = \set{tu : u \in A,\, \norm{u} = R,\, 0 \le t \le 1},
\end{gather*}
and $\pi : W \setminus \set{0} \to S,\, u \mapsto u/\norm{u}$ is the radial projection onto $S$. Let $\Gamma = \{\gamma \in C(X,W) : \gamma(X) \text{ is closed and} \restr{\gamma}{A} = \id[\! A]\}$, and set
\[
c := \inf_{\gamma \in \Gamma}\, \sup_{u \in \gamma(X)}\, I(u).
\]
Then
\begin{equation} \label{5}
\inf I(B) \le c \le \sup I(X),
\end{equation}
in particular, $c$ is finite. If, in addition, $I$ satisfies the {\em \PS{c}} condition, then $c$ is a critical value of $I$.
\end{theorem}

\begin{remark}
The linking construction used in the proof of Theorem \ref{Theorem 3} in \cite{YaPe2} has also been used in Perera and Szulkin \cite{MR2153141} to obtain nontrivial solutions of $p$-Laplacian problems with nonlinearities that cross an eigenvalue. A similar construction based on the notion of cohomological linking was given in Degiovanni and Lancelotti \cite{MR2371112}. See also Perera et al.\! \cite[Proposition 3.23]{MR2640827}.
\end{remark}

Now let $I$ be an even $C^1$-functional defined on a Banach space $W$, and let $\A^\ast$ denote the class of symmetric subsets of $W$. Let $r > 0$, let $S_r = \set{u \in W : \norm{u} = r}$, let $0 < b \le + \infty$, and let $\Gamma$ denote the group of odd homeomorphisms of $W$ that are the identity outside $I^{-1}(0,b)$. The pseudo-index of $M \in \A^\ast$ related to $i$, $S_r$, and $\Gamma$ is defined by
\[
i^\ast(M) = \min_{\gamma \in \Gamma}\, i(\gamma(M) \cap S_r)
\]
(see Benci \cite{MR84c:58014}). We will prove Theorems \ref{Theorem 4} and \ref{Theorem 5} using the following critical point theorem proved in Yang and Perera \cite{YaPe2}, which generalizes Bartolo et al. \cite[Theorem 2.4]{MR713209}.

\begin{theorem} \label{Theorem 6}
Let $A_0$ and $B_0$ be symmetric subsets of $S$ such that $A_0$ is compact, $B_0$ is closed, and
\[
i(A_0) \ge k + m, \qquad i(S \setminus B_0) \le k
\]
for some integers $k \ge 0$ and $m \ge 1$. Assume that there exists $R > r$ such that
\[
\sup I(A) \le 0 < \inf I(B), \qquad \sup I(X) < b,
\]
where $A = \set{Ru : u \in A_0}$, $B = \set{ru : u \in B_0}$, and $X = \set{tu : u \in A,\, 0 \le t \le 1}$. For $j = k + 1,\dots,k + m$, let
\[
\A_j^\ast = \set{M \in \A^\ast : M \text{ is compact and } i^\ast(M) \ge j},
\]
and set
\[
c_j^\ast := \inf_{M \in \A_j^\ast}\, \max_{u \in M}\, I(u).
\]
Then
\[
\inf I(B) \le c_{k+1}^\ast \le \dotsb \le c_{k+m}^\ast \le \sup I(X),
\]
in particular, $0 < c_j^\ast < b$. If, in addition, $I$ satisfies the {\em \PS{c}} condition for all $c \in (0,b)$, then each $c_j^\ast$ is a critical value of $I$ and there are $m$ distinct pairs of associated critical points.
\end{theorem}

\begin{remark}
Constructions similar to the one used in the proof of Theorem \ref{Theorem 6} in \cite{YaPe2} have also been used in Fadell and Rabinowitz \cite{MR57:17677} to prove bifurcation results for Hamiltonian systems and in Perera and Szulkin \cite{MR2153141} to prove multiplicity results for $p$-Laplacian problems. See also Perera et al.\! \cite[Proposition 3.44]{MR2640827}.
\end{remark}

\subsection{Some estimates}

It was shown in \cite[Theorem 3.1.(2)]{MR1695021} that the infimum in \eqref{3} is attained by the family of functions
\[
u_\eps(x) = \frac{C_{N,p,s}\, \eps^{(N-p)/(p-s)\, p}}{\left[\eps + |x|^{(p-s)/(p-1)}\right]^{(N-p)/(p-s)}}, \quad \eps > 0
\]
when $\Omega = \R^N$, where $C_{N,p,s} > 0$ is chosen so that
\[
\int_{\R^N} |\nabla u_\eps|^p\, dx = \int_{\R^N} \frac{u_\eps^{p^\ast(s)}}{|x|^s}\, dx = \mu_s^{(N-s)/(p-s)}.
\]
Take a smooth function $\eta : [0,\infty) \to [0,1]$ such that $\eta(s) = 1$ for $s \le 1/4$ and $\eta(s) = 0$ for $s \ge 1/2$, and set
\[
u_{\eps,\delta}(x) = \eta\!\left(\frac{|x|}{\delta}\right) u_\eps(x), \quad v_{\eps,\delta}(x) = \frac{u_{\eps,\delta}(x)}{\left(\dint_{\R^N} \frac{u_{\eps,\delta}^{p^\ast(s)}}{|x|^s}\, dx\right)^{1/p^\ast(s)}}, \quad \eps, \delta > 0,
\]
so that
\begin{equation} \label{6}
\int_{\R^N} \frac{v_{\eps,\delta}^{p^\ast(s)}}{|x|^s}\, dx = 1.
\end{equation}
The following estimates were obtained in \cite[Lemma 11.1.(1),(3),(4)]{MR1695021}:
\begin{gather}
\label{7} \int_{\R^N} |\nabla v_{\eps,\delta}|^p\, dx \le \mu_s + C \eps^{(N-p)/(p-s)},\\[10pt]
\label{8} \int_{\R^N} v_{\eps,\delta}^p\, dx \ge \begin{cases}
\dfrac{1}{C}\, \eps^{(p-1)\, p/(p-s)} & \text{if } N > p^2\\[10pt]
\dfrac{1}{C}\, \eps^{(p-1)\, p/(p-s)}\, |\!\log \eps| & \text{if } N = p^2,
\end{cases}
\end{gather}
where $C = C(N,p,s,\delta) > 0$ is a constant. While these estimates are sufficient for the proof of Theorem \ref{Theorem 1}, we will need the following finer estimates in order to prove Theorem \ref{Theorem 2}.

\begin{lemma}
There exists a constant $C = C(N,p,s) > 0$ such that
\begin{gather}
\label{21} \int_{\R^N} |\nabla v_{\eps,\delta}|^p\, dx \le \mu_s + C \Theta_{\eps,\delta}^{(N-p)/(p-s)},\\[10pt]
\label{22} \int_{\R^N} v_{\eps,\delta}^p\, dx \ge \begin{cases}
\dfrac{1}{C}\, \eps^{(p-1)\, p/(p-s)} & \text{if } N > p^2\\[10pt]
\dfrac{1}{C}\, \eps^{(p-1)\, p/(p-s)}\, |\!\log \Theta_{\eps,\delta}| & \text{if } N = p^2,
\end{cases}
\end{gather}
where $\Theta_{\eps,\delta} = \eps\, \delta^{-(p-s)/(p-1)}$.
\end{lemma}

\begin{proof}
We have
\[
u_{\eps,\delta}(\delta x) = \delta^{-(N-p)/p}\, u_{\Theta_{\eps,\delta},1}(x)
\]
and
\[
\int_{\R^N} \frac{u_{\eps,\delta}^{p^\ast(s)}}{|x|^s}\, dx = \int_{\R^N} \frac{u_{\Theta_{\eps,\delta},1}^{p^\ast(s)}}{|x|^s}\, dx.
\]
So
\[
v_{\eps,\delta}(\delta x) = \delta^{-(N-p)/p}\, v_{\Theta_{\eps,\delta},1}(x)
\]
and hence
\[
\nabla v_{\eps,\delta}(\delta x) = \delta^{-N/p}\, \nabla v_{\Theta_{\eps,\delta},1}(x).
\]
Then
\[
\int_{\R^N} |\nabla v_{\eps,\delta}(x)|^p\, dx = \delta^N \int_{\R^N} |\nabla v_{\eps,\delta}(\delta x)|^p\, dx = \int_{\R^N} |\nabla v_{\Theta_{\eps,\delta},1}(x)|^p\, dx
\]
and
\[
\int_{\R^N} v_{\eps,\delta}^p(x)\, dx = \delta^N \int_{\R^N} v_{\eps,\delta}^p(\delta x)\, dx = \delta^p \int_{\R^N} v_{\Theta_{\eps,\delta},1}^p(x)\, dx,
\]
so \eqref{21} and \eqref{22} follow from \eqref{7} and \eqref{8}, respectively.
\end{proof}

Let $i$, $\M$, $\Psi$, and $\lambda_k$ be as in the introduction, and suppose that $\lambda_k < \lambda_{k+1}$. Then the sublevel set $\Psi^{\lambda_k}$ has a compact symmetric subset $E$ of index $k$ that is bounded in $L^\infty(\Omega) \cap C^{1,\alpha}_\loc(\Omega)$ (see Degiovanni and Lancelotti \cite[Theorem 2.3]{MR2514055}). Let $\delta_0 = \dist{0}{\bdry{\Omega}}$, take a smooth function $\theta : [0,\infty) \to [0,1]$ such that $\theta(s) = 0$ for $s \le 3/4$ and $\theta(s) = 1$ for $s \ge 1$, and set
\[
v_\delta(x) = \theta\!\left(\frac{|x|}{\delta}\right) v(x), \quad v \in E,\, 0 < \delta \le \frac{\delta_0}{2}.
\]
Since $E \subset \Psi^{\lambda_k}$ is bounded in $C^1(B_{\delta_0/2}(0))$,
\begin{equation} \label{9}
\int_\Omega |\nabla v_\delta|^p\, dx \le \int_{\Omega \setminus B_\delta(0)} |\nabla v|^p\, dx + C \int_{B_\delta(0)} \left(|\nabla v|^p + \frac{|v|^p}{\delta^p}\right) dx \le 1 + C \delta^{N-p}
\end{equation}
and
\begin{equation} \label{10}
\int_\Omega |v_\delta|^p\, dx \ge \int_{\Omega \setminus B_\delta(0)} |v|^p\, dx = \int_\Omega |v|^p\, dx - \int_{B_\delta(0)} |v|^p\, dx \ge \frac{1}{\lambda_k} - C \delta^N,
\end{equation}
where $C = C(N,p,s,\Omega,k) > 0$ is a constant. By \eqref{28} and \eqref{10},
\begin{equation} \label{11}
\int_\Omega \frac{|v_\delta|^{p^\ast(s)}}{|x|^s}\, dx \ge \frac{1}{C}
\end{equation}
if $\delta > 0$ is sufficiently small.

Now let $\pi : W^{1,p}_0(\Omega) \setminus \set{0} \to \M,\, u \mapsto u/\norm{u}$ be the radial projection onto $\M$, and set
\[
w = \pi(v_\delta), \quad v \in E.
\]
If $\delta > 0$ is sufficiently small,
\begin{equation} \label{12}
\Psi(w) = \frac{\dint_\Omega |\nabla v_\delta|^p\, dx}{\dint_\Omega |v_\delta|^p\, dx} \le \lambda_k + C \delta^{N-p} < \lambda_{k+1}
\end{equation}
by \eqref{9} and \eqref{10}, and
\begin{equation} \label{13}
\int_\Omega \frac{|w|^{p^\ast(s)}}{|x|^s}\, dx = \frac{\dint_\Omega \frac{|v_\delta|^{p^\ast(s)}}{|x|^s}\, dx}{\left(\dint_\Omega |\nabla v_\delta|^p\, dx\right)^{p^\ast(s)/p}} \ge \frac{1}{C}
\end{equation}
by \eqref{9} and \eqref{11}. Since $\supp w = \supp v_\delta \subset \Omega \setminus B_{3 \delta/4}(0)$ and $\supp \pi(v_{\eps,\delta}) = \supp v_{\eps,\delta} \subset \closure{B_{\delta/2}(0)}$,
\begin{equation} \label{14}
\supp w \cap \supp \pi(v_{\eps,\delta}) = \emptyset.
\end{equation}
Set
\[
E_\delta = \set{w : v \in E}.
\]

\begin{lemma} \label{Lemma 1}
For all sufficiently small $\delta > 0$,
\begin{enumroman}
\item \label{15} $E_\delta \cap \Psi_{\lambda_{k+1}} = \emptyset$,
\item \label{16} $i(E_\delta) = k$,
\item \label{17} $\pi(v_{\eps,\delta}) \notin E_\delta$.
\end{enumroman}
\end{lemma}

\begin{proof}
\ref{15} follows from \eqref{12}. By \ref{15}, $E_\delta \subset \M \setminus \Psi_{\lambda_{k+1}}$ and hence
\[
i(E_\delta) \le i(\M \setminus \Psi_{\lambda_{k+1}}) = k
\]
by the monotonicity of the index and \eqref{4}. On the other hand, since $E \to E_\delta,\, v \mapsto \pi(v_\delta)$ is an odd continuous map,
\[
i(E_\delta) \ge i(E) = k.
\]
\ref{16} follows. \ref{17} is immediate from \eqref{14}.
\end{proof}

\section{Proofs}

\subsection{Proof of Theorem \ref{Th4}}

All nontrivial critical points of $I_\lambda$ lie on the Nehari manifold
\[
\calN = \set{u \in W^{1,p}_0(\Omega) \setminus \set{0} : I_\lambda'(u)\, u = 0}.
\]
We will show that $I_\lambda$ attains the ground state energy
\[
c := \inf_{u \in \calN}\, I_\lambda(u)
\]
at a positive critical point.

Since $0 < \lambda < \lambda_1$, $\calN$ is closed, bounded away from the origin, and for $u \in W^{1,p}_0(\Omega) \setminus \set{0}$ and $t > 0$, $tu \in \calN$ if and only if $t = t_u$, where
\[
t_u = \left[\frac{\dint_\Omega \big(|\nabla u|^p - \lambda\, |u|^p\big)\, dx}{\dint_\Omega \frac{|u|^{p^\ast(s)}}{|x|^s}\, dx}\right]^{(N-p)/(p-s)\, p}.
\]
Moreover,
\[
I_\lambda(t_u u) = \sup_{t > 0}\, I_\lambda(tu) = \frac{p - s}{(N - s)\, p}\; \psi_\lambda(u)^{(N-s)/(p-s)},
\]
where
\[
\psi_\lambda(u) = \frac{\dint_\Omega \big(|\nabla u|^p - \lambda\, |u|^p\big)\, dx}{\left(\dint_\Omega \frac{|u|^{p^\ast(s)}}{|x|^s}\, dx\right)^{p/p^\ast(s)}}.
\]
By \eqref{6}--\eqref{8},
\[
\psi_\lambda(v_{\eps,\delta}) \le \begin{cases}
\mu_s - \dfrac{\eps^{(p-1)\, p/(p-s)}}{C} + C \eps^{(N-p)/(p-s)} & \text{if } N > p^2\\[10pt]
\mu_s - \dfrac{\eps^{(p-1)\, p/(p-s)}}{C}\; |\!\log \eps| + C \eps^{(p-1)\, p/(p-s)} & \text{if } N = p^2,
\end{cases}
\]
and in both cases the last expression is strictly less than $\mu_s$ if $\eps > 0$ is sufficiently small, so
\[
c \le I_\lambda(t_{v_{\eps,\delta}} v_{\eps,\delta}) < \frac{p - s}{(N - s)\, p}\; \mu_s^{(N-s)/(p-s)}.
\]
Then $I_\lambda$ satisfies the \PS{c} condition by \cite[Theorem 4.1.(2)]{MR1695021}, and hence $\restr{I_\lambda}{\calN}$ has a minimizer $u_0$ by a standard argument. Then $|u_0|$ is also a minimizer, which is positive by the strong maximum principle.

\subsection{Proof of Theorem \ref{Theorem 1}}

We will show that problem \eqref{1} has a nontrivial solution as long as $\lambda > \lambda_1$ is not an eigenvalue from the sequence $\seq{\lambda_k}$. Then we have $\lambda_k < \lambda < \lambda_{k+1}$ for some $k \in \N$. Fix $\delta > 0$ so small that the first inequality in \eqref{12} implies
\begin{equation} \label{18}
\Psi(w) \le \lambda \quad \forall w \in E_\delta
\end{equation}
and the conclusions of Lemma \ref{Lemma 1} hold. Then let $A_0 = E_\delta$ and $B_0 = \Psi_{\lambda_{k+1}}$, and note that $A_0$ and $B_0$ are disjoint nonempty closed symmetric subsets of $\M$ such that
\begin{equation} \label{23}
i(A_0) = i(\M \setminus B_0) = k
\end{equation}
by Lemma \ref{Lemma 1} \ref{15}, \ref{16} and \eqref{4}. Now let $R > r > 0$, let $v_0 = \pi(v_{\eps,\delta})$, which is in $\M \setminus A_0$ by Lemma \ref{Lemma 1} \ref{17}, and let $A$, $B$ and $X$ be as in Theorem \ref{Theorem 3}.

For $u \in B_0$,
\[
I_\lambda(ru) \ge \frac{1}{p} \left(1 - \frac{\lambda}{\lambda_{k+1}}\right) r^p - \frac{r^{p^\ast(s)}}{p^\ast(s)\, \mu_s^{p^\ast(s)/p}}.
\]
Since $\lambda < \lambda_{k+1}$, and $s < p$ implies $p^\ast(s) > p$, it follows that $\inf I_\lambda(B) > 0$ if $r$ is sufficiently small.

Next we show that $I_\lambda \le 0$ on $A$ if $R$ is sufficiently large. For $w \in A_0$ and $t \ge 0$,
\[
I_\lambda(tw) \le \frac{t^p}{p} \left(1 - \frac{\lambda}{\Psi(w)}\right) \le 0
\]
by \eqref{18}. Now let $w \in A_0$ and $0 \le t \le 1$, and set $u = \pi((1 - t)\, w + tv_0)$. Clearly, $\norm{(1 - t)\, w + tv_0} \le 1$, and since the supports of $w$ and $v_0$ are disjoint by \eqref{14},
\[
\int_\Omega \frac{|(1 - t)\, w + tv_0|^{p^\ast(s)}}{|x|^s}\, dx = (1 - t)^{p^\ast(s)} \int_\Omega \frac{|w|^{p^\ast(s)}}{|x|^s}\, dx + t^{p^\ast(s)} \int_\Omega \frac{v_0^{p^\ast(s)}}{|x|^s}\, dx.
\]
In view of \eqref{13}, and since
\[
\int_\Omega \frac{v_0^{p^\ast(s)}}{|x|^s}\, dx = \frac{\dint_\Omega \frac{v_{\eps,\delta}^{p^\ast(s)}}{|x|^s}\, dx}{\left(\dint_\Omega |\nabla v_{\eps,\delta}|^p\, dx\right)^{p^\ast(s)/p}} \ge \frac{1}{C}
\]
by \eqref{6} and \eqref{7} if $\eps > 0$ is sufficiently small, it follows that
\[
\int_\Omega \frac{|u|^{p^\ast(s)}}{|x|^s}\, dx = \frac{\dint_\Omega \frac{|(1 - t)\, w + tv_0|^{p^\ast(s)}}{|x|^s}\, dx}{\norm{(1 - t)\, w + tv_0}^{p^\ast(s)}} \ge \frac{1}{C}.
\]
Then
\[
I_\lambda(Ru) \le \frac{R^p}{p} - \frac{R^{p^\ast(s)}}{p^\ast(s)} \int_\Omega \frac{|u|^{p^\ast(s)}}{|x|^s}\, dx \le 0
\]
if $R$ is sufficiently large.

Now we show that
\begin{equation} \label{19}
\sup I_\lambda(X) < \frac{p - s}{(N - s)\, p}\; \mu_s^{(N-s)/(p-s)}
\end{equation}
if $\eps > 0$ is sufficiently small. Noting that
\[
X = \set{\rho\, \pi((1 - t)\, w + tv_0) : w \in E_\delta,\, 0 \le t \le 1,\, 0 \le \rho \le R},
\]
let $w \in E_\delta$ and $0 \le t \le 1$, and set $u = \pi((1 - t)\, w + tv_0)$. Then
\begin{eqnarray}
\sup_{0 \le \rho \le R}\, I_\lambda(\rho u) & \le & \sup_{\rho \ge 0}\, \left[\frac{\rho^p}{p} \left(1 - \lambda \int_\Omega |u|^p\, dx\right) - \frac{\rho^{p^\ast(s)}}{p^\ast(s)} \int_\Omega \frac{|u|^{p^\ast(s)}}{|x|^s}\, dx\right] \notag\\[10pt]
\label{20} & = & \frac{p - s}{(N - s)\, p}\; \psi_\lambda(u)^{(N-s)/(p-s)},
\end{eqnarray}
where
\begin{eqnarray}
\psi_\lambda(u) & = & \frac{\left(1 - \lambda \dint_\Omega |u|^p\, dx\right)^+}{\left(\dint_\Omega \frac{|u|^{p^\ast(s)}}{|x|^s}\, dx\right)^{p/p^\ast(s)}} \notag\\[10pt]
& = & \frac{\left(\dint_\Omega \Big[|(1 - t)\, \nabla w + t\, \nabla v_0|^p - \lambda\, |(1 - t)\, w + tv_0|^p\Big]\, dx\right)^+}{\left(\dint_\Omega \frac{|(1 - t)\, w + tv_0|^{p^\ast(s)}}{|x|^s}\, dx\right)^{p/p^\ast(s)}} \notag\\[10pt]
\label{24} & \le & \frac{(1 - t)^p \left(1 - \lambda \dint_\Omega |w|^p\, dx\right)^+ + t^p \left(1 - \lambda \dint_\Omega v_0^p\, dx\right)^+}{\left((1 - t)^{p^\ast(s)} \dint_\Omega \frac{|w|^{p^\ast(s)}}{|x|^s}\, dx + t^{p^\ast(s)} \dint_\Omega \frac{v_0^{p^\ast(s)}}{|x|^s}\, dx\right)^{p/p^\ast(s)}}
\end{eqnarray}
since the supports of $w$ and $v_0$ are disjoint. Since
\[
1 - \lambda \int_\Omega |w|^p\, dx = 1 - \frac{\lambda}{\Psi(w)} \le 0
\]
by \eqref{18},
\begin{eqnarray*}
\psi_\lambda(u) & \le & \psi_\lambda(v_0)\\[10pt]
& = & \frac{\left(\dint_\Omega \Big[|\nabla v_{\eps,\delta}|^p - \lambda\, v_{\eps,\delta}^p\Big]\, dx\right)^+}{\left(\dint_\Omega \frac{v_{\eps,\delta}^{p^\ast(s)}}{|x|^s}\, dx\right)^{p/p^\ast(s)}}\\[10pt]
& \le & \begin{cases}
\mu_s - \dfrac{\eps^{(p-1)\, p/(p-s)}}{C} + C \eps^{(N-p)/(p-s)} & \text{if } N > p^2\\[10pt]
\mu_s - \dfrac{\eps^{(p-1)\, p/(p-s)}}{C}\; |\!\log \eps| + C \eps^{(p-1)\, p/(p-s)} & \text{if } N = p^2
\end{cases}
\end{eqnarray*}
by \eqref{6}--\eqref{8}. In both cases the last expression is strictly less than $\mu_s$ if $\eps > 0$ is sufficiently small, so \eqref{19} follows from \eqref{20}.

The inequalities \eqref{5} now imply that
\[
0 < c < \frac{p - s}{(N - s)\, p}\; \mu_s^{(N-s)/(p-s)}.
\]
Then $I_\lambda$ satisfies the \PS{c} condition by \cite[Theorem 4.1.(2)]{MR1695021}, and hence $c$ is a positive critical value of $I_\lambda$ by Theorem \ref{Theorem 3}.

\subsection{Proof of Theorem \ref{Theorem 2}}

The case where $\lambda > \lambda_1$ is an eigenvalue, but not from the sequence $\seq{\lambda_k}$, was covered in the proof of Theorem \ref{Theorem 1}, so we may assume that $\lambda = \lambda_k < \lambda_{k+1}$ for some $k \in \N$. Take $\delta > 0$ so small that \eqref{12} and the conclusions of Lemma \ref{Lemma 1} hold, let $A_0$, $B_0$ and $v_0$ be as in the proof of Theorem \ref{Theorem 1}, and let $A$, $B$ and $X$ be as in Theorem \ref{Theorem 3}.

As before, $\inf I_\lambda(B) > 0$ if $r$ is sufficiently small, and
\[
I_\lambda(R\, \pi((1 - t)\, w + tv_0)) \le 0 \quad \forall w \in A_0,\, 0 \le t \le 1
\]
if $\Theta_{\eps,\delta}$ is sufficiently small and $R$ is sufficiently large. On the other hand,
\[
I_\lambda(tw) \le \frac{t^p}{p} \left(1 - \frac{\lambda_k}{\Psi(w)}\right) \le C R^p \delta^{N-p} \quad \forall w \in A_0,\, 0 \le t \le R
\]
by \eqref{12}. It follows that $\sup I_\lambda(A) < \inf I_\lambda(B)$ if $\delta$ is also sufficiently small.

It only remains to verify \eqref{19} for suitably small $\eps$ and $\delta$. Maximizing the last expression in \eqref{24} over $0 \le t \le 1$ gives
\begin{equation} \label{25}
\psi_\lambda(u) \le \left[\psi_\lambda(v_0)^{(N-s)/(p-s)} + \psi_\lambda(w)^{(N-s)/(p-s)}\right]^{(p-s)/(N-s)}.
\end{equation}
By \eqref{6}, \eqref{21}, and \eqref{22},
\begin{equation}
\psi_\lambda(v_0) = \frac{\left(\dint_\Omega \Big[|\nabla v_{\eps,\delta}|^p - \lambda_k\, v_{\eps,\delta}^p\Big]\, dx\right)^+}{\left(\dint_\Omega \frac{v_{\eps,\delta}^{p^\ast(s)}}{|x|^s}\, dx\right)^{p/p^\ast(s)}} \le \mu_s - \frac{\eps^{(p-1)\, p/(p-s)}}{C} + C \Theta_{\eps,\delta}^{(N-p)/(p-s)},
\end{equation}
and by \eqref{12} and \eqref{13},
\begin{equation} \label{26}
\psi_\lambda(w) = \frac{\left(1 - \dfrac{\lambda_k}{\Psi(w)}\right)^+}{\left(\dint_\Omega \frac{|w|^{p^\ast(s)}}{|x|^s}\, dx\right)^{p/p^\ast(s)}} \le C \delta^{N-p}.
\end{equation}
Recalling that $\Theta_{\eps,\delta} = \eps\, \delta^{-(p-s)/(p-1)}$, if there exist $\alpha \in (0,(p - 1)/(p - s))$ and a sequence $\eps_j \to 0$ such that, for $\eps = \eps_j$ and $\delta = \eps_j^\alpha$, $\psi_\lambda(v_0) < \mu_s/3$, then $\psi_\lambda(u) \le 2 \mu_s/3$ for sufficiently large $j$ by \eqref{25} and \eqref{26}, which together with \eqref{20} gives the desired result. So we may assume that for all $\alpha \in (0,(p - 1)/(p - s))$, $\psi_\lambda(v_0) \ge \mu_s/3$ for all sufficiently small $\eps$ and $\delta = \eps^\alpha$. Since $(p - s)/(N - s) < 1$, then \eqref{25}--\eqref{26} with $\delta = \eps^\alpha$ yield
\begin{eqnarray*}
\psi_\lambda(u) & \le & \psi_\lambda(v_0) \left[1 + \left(\frac{\psi_\lambda(w)}{\psi_\lambda(v_0)}\right)^{(N-s)/(p-s)}\right]\\[5pt]
& \le & \psi_\lambda(v_0) + C\, \psi_\lambda(w)^{(N-s)/(p-s)}\\[10pt]
& \le & \mu_s - \eps^{(p-1)\, p/(p-s)} \left[\frac{1}{C} - C \eps^{(N-p)(N-s)(\alpha - \alpha_1)/(p-s)} - C \eps^{(N-p)(\alpha_2 - \alpha)/(p-1)}\right],
\end{eqnarray*}
where
\[
0 < \alpha_1 := \frac{(p - 1)\, p}{(N - p)(N - s)} < \frac{(N - p^2)(p - 1)}{(N - p)(p - s)} =: \alpha_2 < \frac{p - 1}{p - s}
\]
by \eqref{27}. Taking $\alpha \in (\alpha_1,\alpha_2)$ now gives the desired conclusion.

\subsection{Proofs of Theorems \ref{Theorem 4} and \ref{Theorem 5}}

We only give the proof of Theorem \ref{Theorem 5}. Proof of Theorem \ref{Theorem 4} is similar and simpler. By \cite[Theorem 4.1.(2)]{MR1695021}, $I_\lambda$ satisfies the \PS{c} condition for all
\[
c < \frac{p - s}{(N - s)\, p}\; \mu_s^{(N-s)/(p-s)},
\]
so we apply Theorem \ref{Theorem 6} with $b$ equal to the right-hand side.

By Degiovanni and Lancelotti \cite[Theorem 2.3]{MR2514055}, the sublevel set $\Psi^{\lambda_{k+m}}$ has a compact symmetric subset $A_0$ with
\[
i(A_0) = k + m.
\]
We take $B_0 = \Psi_{\lambda_{k+1}}$, so that
\[
i(\M \setminus B_0) = k
\]
by \eqref{4}. Let $R > r > 0$ and let $A$, $B$ and $X$ be as in Theorem \ref{Theorem 6}. For $u \in \Psi_{\lambda_{k+1}}$,
\[
I_\lambda(ru) \ge \frac{r^p}{p} \left(1 - \frac{\lambda}{\lambda_{k+1}}\right) - \frac{r^{p^\ast(s)}}{p^\ast(s)\, \mu_s^{p^\ast(s)/p}}
\]
by \eqref{3}. Since $\lambda < \lambda_{k+1}$, and $s < p$ implies $p^\ast(s) > p$, it follows that $\inf I_\lambda(B) > 0$ if $r$ is sufficiently small. For $u \in A_0 \subset \Psi^{\lambda_{k+1}}$,
\[
I_\lambda(Ru) \le \frac{R^p}{p} \left(1 - \frac{\lambda}{\lambda_{k+1}}\right) - \frac{R^{p^\ast(s)}}{p^\ast(s)\, \lambda_{k+1}^{p^\ast(s)/p}\, V_s(\Omega)^{(p-s)/(N-p)}}
\]
by \eqref{28}, so there exists $R > r$ such that $I_\lambda \le 0$ on $A$. For $u \in X$,
\begin{align*}
I_\lambda(u) & \le \frac{\lambda_{k+1} - \lambda}{p} \int_\Omega |u|^p\, dx - \frac{1}{p^\ast(s)\, V_s(\Omega)^{(p-s)/(N-p)}} \left(\int_\Omega |u|^p\, dx\right)^{p^\ast(s)/p}\\[10pt]
& \le \sup_{\rho \ge 0}\, \left[\frac{(\lambda_{k+1} - \lambda)\, \rho}{p} - \frac{\rho^{p^\ast(s)/p}}{p^\ast(s)\, V_s(\Omega)^{(p-s)/(N-p)}}\right]\\[10pt]
& = \frac{p - s}{(N - s)\, p}\; (\lambda_{k+1} - \lambda)^{(N-s)/(p-s)}\, V_s(\Omega).
\end{align*}
So
\[
\sup I_\lambda(X) \le \frac{p - s}{(N - s)\, p}\; (\lambda_{k+1} - \lambda)^{(N-s)/(p-s)}\, V_s(\Omega) < \frac{p - s}{(N - s)\, p}\; \mu_s^{(N-s)/(p-s)}
\]
by \eqref{1.10}. Theorem \ref{Theorem 6} now gives $m$ distinct pairs of (nontrivial) critical points $\pm\, u^\lambda_j,\, j = 1,\dots,m$ of $I_\lambda$ such that
\begin{equation} \label{4.4}
0 < I_\lambda(u^\lambda_j) \le \frac{p - s}{(N - s)\, p}\; (\lambda_{k+1} - \lambda)^{(N-s)/(p-s)}\, V_s(\Omega) \to 0 \text{ as } \lambda \nearrow \lambda_{k+1}.
\end{equation}
Then
\[
\int_\Omega \frac{|u^\lambda_j|^{p^\ast(s)}}{|x|^s}\, dx = \frac{(N - s)\, p}{p - s}\, \left[I_\lambda(u^\lambda_j) - \frac{1}{p}\, I_\lambda'(u^\lambda_j)\, u^\lambda_j\right] = \frac{(N - s)\, p}{p - s}\; I_\lambda(u^\lambda_j) \to 0
\]
and hence $u^\lambda_j \to 0$ in $L^p(\Omega)$ also by \eqref{28}, so
\[
\int_\Omega |\nabla u^\lambda_j|^p\, dx = p\, I_\lambda(u^\lambda_j) + \lambda \int_\Omega |u^\lambda_j|^p\, dx + \frac{p}{p^\ast(s)} \int_\Omega \frac{|u^\lambda_j|^{p^\ast(s)}}{|x|^s}\, dx \to 0.
\]
This completes the proof of Theorem \ref{Theorem 5}.

\def\cdprime{$''$}


\begin{thebibliography}{10}

\bibitem{MR713209}
P.~Bartolo, V.~Benci, and D.~Fortunato.
\newblock Abstract critical point theorems and applications to some nonlinear
  problems with ``strong'' resonance at infinity.
\newblock {\em Nonlinear Anal.}, 7(9):981--1012, 1983.

\bibitem{MR84c:58014}
Vieri Benci.
\newblock On critical point theory for indefinite functionals in the presence
  of symmetries.
\newblock {\em Trans. Amer. Math. Soc.}, 274(2):533--572, 1982.

\bibitem{MR2371112}
Marco Degiovanni and Sergio Lancelotti.
\newblock Linking over cones and nontrivial solutions for {$p$}-{L}aplace
  equations with {$p$}-superlinear nonlinearity.
\newblock {\em Ann. Inst. H. Poincar\'e Anal. Non Lin\'eaire}, 24(6):907--919,
  2007.

\bibitem{MR2514055}
Marco Degiovanni and Sergio Lancelotti.
\newblock Linking solutions for {$p$}-{L}aplace equations with nonlinearity at
  critical growth.
\newblock {\em J. Funct. Anal.}, 256(11):3643--3659, 2009.

\bibitem{MR57:17677}
Edward~R. Fadell and Paul~H. Rabinowitz.
\newblock Generalized cohomological index theories for {L}ie group actions with
  an application to bifurcation questions for {H}amiltonian systems.
\newblock {\em Invent. Math.}, 45(2):139--174, 1978.

\bibitem{MR1695021}
N.~Ghoussoub and C.~Yuan.
\newblock Multiple solutions for quasi-linear {PDE}s involving the critical
  {S}obolev and {H}ardy exponents.
\newblock {\em Trans. Amer. Math. Soc.}, 352(12):5703--5743, 2000.

\bibitem{MR3530213}
Sunra Mosconi, Kanishka Perera, Marco Squassina, and Yang Yang.
\newblock The {B}rezis--{N}irenberg problem for the fractional p-{L}aplacian.
\newblock {\em Calc. Var. Partial Differential Equations}, 55(4):55:105, 2016.

\bibitem{MR1998432}
Kanishka Perera.
\newblock Nontrivial critical groups in {$p$}-{L}aplacian problems via the
  {Y}ang index.
\newblock {\em Topol. Methods Nonlinear Anal.}, 21(2):301--309, 2003.

\bibitem{MR2640827}
Kanishka Perera, Ravi~P. Agarwal, and Donal O'Regan.
\newblock {\em Morse theoretic aspects of {$p$}-{L}aplacian type operators},
  volume 161 of {\em Mathematical Surveys and Monographs}.
\newblock American Mathematical Society, Providence, RI, 2010.

\bibitem{PeSqYa1}
Kanishka Perera, Marco Squassina, and Yang Yang.
\newblock Bifurcation and multiplicity results for critical $p$-{L}aplacian
  problems.
\newblock {\em Topol. Methods Nonlinear Anal.}, 47(1):187--194, 2016.

\bibitem{MR2153141}
Kanishka Perera and Andrzej Szulkin.
\newblock {$p$}-{L}aplacian problems where the nonlinearity crosses an
  eigenvalue.
\newblock {\em Discrete Contin. Dyn. Syst.}, 13(3):743--753, 2005.

\bibitem{MR0488128}
Paul~H. Rabinowitz.
\newblock Some critical point theorems and applications to semilinear elliptic
  partial differential equations.
\newblock {\em Ann. Scuola Norm. Sup. Pisa Cl. Sci. (4)}, 5(1):215--223, 1978.

\bibitem{YaPe2}
Yang Yang and Kanishka Perera.
\newblock ${N}$-{L}aplacian problems with critical {T}rudinger-{M}oser
  nonlinearities.
\newblock Ann. Sc. Norm. Super. Pisa Cl. Sci. (5), to appear,
  \href{http://arxiv.org/abs/1406.6242}{\tt arXiv:1406.6242 [math.AP]}.

\end{thebibliography}
\end{document}